\documentclass[11pt,reqno,english,a4]{amsart}
\usepackage{amsmath,amsthm,amsfonts,amssymb,amscd}
\usepackage[latin1]{inputenc}
\usepackage{psfrag}
\usepackage{epsfig}
\usepackage{a4wide}
\usepackage[]{graphicx}
\usepackage{times}
\usepackage{bm}
\usepackage{mathdots}
\usepackage{mathrsfs}

\newtheorem{theorem}{Theorem}[section]
\newtheorem{proposition}[theorem]{Proposition}
\newtheorem{lemma}[theorem]{Lemma}
\newtheorem{corollary}[theorem]{Corollary}

\newtheorem{remark}[theorem]{Remark}

\newcommand{\R}{\mathbb{R}}

\newcommand{\Pn}{\mathbb{P}}

\newcommand{\Aff}{\mathbb{A}\mathrm{ff}}

\newcommand{\el}{\mathscr{L}}

\newcommand{\Conf}{\mathrm{Conf}}

\newcommand{\Iso}{\mathrm{Iso}}

\newcommand{\Sn}{\mathbb{S}}
\newcommand{\En}{\mathbb{E}}
\newcommand{\Hn}{\mathbb{H}}
\newcommand{\g}{\mathfrak{g}}
\newcommand{\h}{\mathfrak{h}}

\newcommand{\so}{\mathfrak{so}}

\newcommand{\Nu}{\mathfrak{N}}
\newcommand{\gi}{\mathsf{g}}

\begin{document}
	
	\title[Cohomogeneity one conformal actions on the three-dimensional essential Riemannian spaces]
	{Cohomogeneity one conformal actions on the three-dimensional essential Riemannian spaces}

\author{P. Ahmadi}	
\author{M. Hassani}

\thanks{}



\date{\today}
\address{
	P. Ahmadi\\
	Departmental of mathematics\\
	University of Zanjan\\
	University blvd.\\
	Zanjan\\
	Iran}
\email{p.ahmadi@znu.ac.ir}

\address{
	M. Hassani\\
	Departmental of mathematics\\
	University of Zanjan\\
	University blvd.\\
	Zanjan\\
	Iran}
\email{masoud.hasani@znu.ac.ir}

\begin{abstract}
The aim of this paper is to classify the cohomogeneity one conformal actions on the  three-dimensional essential Riemannian spaces, up to orbit equivalence. Among other results, the representations of all connected Lie groups acting with cohomogeneity one or zero within the full conformal group of a given three-dimensional essential Riemannian space are determined, up to conjugacy, and the orbits are specified up to conformal equivalence.
\end{abstract}

\maketitle
\tableofcontents
\medskip
\medskip

\thispagestyle{empty}

%%%%%%%%%%%%%%%%%%%%%%%%%%%%%%%%%%%%%%%%
%%%%%%%%%%%%%%%%%%%%%%%%%%%%%%%%%%%%%%%%
%\tableofcontents

\section{Introduction and Preliminaries}\label{sec.itro}
An interesting problem in the literature of $G$-manifolds is to determine a list of groups admitting actions of the type under investigation. A number of researchers have tried to determine the groups which act isometrically on pseudo-Riemannian manifolds up to local isomorphism or at most up to isomorphism (see for example \cite{Adams1, Kow, Zeg, Zim}). Their aim was not to study the induced orbits and so they did not consider cohomogeneity assumption as a dynamical restriction. An action of a Lie group $G$ on an $n$-manifold $M$ is of \textbf{cohomogeneity $k$}, where $0\leqslant k\leqslant n$, if the maximum dimension among the induced orbits in $M$ is $n-k$. There are few papers about cohomogeneity one Lorentzian manifolds (see for instance \cite{AK1, AK, A3, BDV, vanei1, vanei2}), and many papers in the field of cohomogeneity one Riemannian manifolds (see \cite{Alee, AA, BD, BT, Br, MK, PT, PS, S}). In the later case, the common hypothesis is that the acting group is a closed Lie subgroup of the isometry group of the Riemannian manifold. This assumption causes an strong dynamical restriction, that is the action should be proper. In fact, it is proved that the action of $G$ on $M$ is proper if and only if there is a complete $G$-invariant Riemannian metric on $M$ (see \cite {Alee}). This theorem makes a link between proper actions and Riemannian $G$-manifolds. The orbits of a proper action are closed submanifolds, the isotropy subgroups are compact and the orbit space is Hausdorff (\cite{Adams1}). However, the closeness of the acting group in the conformal Lie group of a Riemannian manifold does not imply the properness of the action (see Proposition \ref{pro} and Theorem \ref{thm.p.1}). Although, conformal actions have been studied by many mathematicians, amount of results in the area of cohomogeneity one conformal actions is not satisfying (see \cite{Has, HA, MN}).

Given a pseudo-Riemannian manifold $(M,\gi)$ the corresponding \textbf{conformal class}, denoted by $[\gi]$, is the set of metrics on $M$ of the form $e^f\gi$ where $f:M\rightarrow \R$ is a smooth function. A diffeomorphism $\varphi:M\rightarrow M$ is said to be conformal if $\varphi$ preserves the conformal class $[\gi]$, i.e., the pullback metric $\varphi^*\gi$ belongs to $[\gi]$. A conformal structure $(M,[\gi])$ is called \textbf{essential} when $\Conf(M,[\gi])$, the group of conformal transformations of $(M,[\gi])$, does not preserve any metric in the metric class $[\gi]$.  Roughly speaking, $(M,[\gi])$ is essential if its conformal group is strictly "bigger" than the isometry group of every metric $\gi\in [\gi]$.

In 1996, J. Ferrand classified the essential Riemannian spaces by proving:
\begin{theorem}
Let $(M,\gi)$ be an essential Riemannian structure of dimension $n\geq 2$. Then $(M,\gi)$ is conformally diffeomorphic to
\begin{itemize}
\item $(\Sn^n,[\gi_{\Sn}])$ if $M$ is compact;
\item $(\En^n,[\gi_{\En}])$ if $M$ is not compact.
\end{itemize}
where $\gi_{\Sn}$ and $\gi_{\En}$ are the canonical metrics on the $n$-dimensional sphere $\Sn^n$ and Euclidean space $\En^n$, respectively.
\end{theorem}

The $n$-dimensional sphere $\Sn^n$ compactifies conformally the Euclidean space $\En^n$ by stereographic map. More precisely, for an arbitrary point $p\in \Sn^n$, the complement $\Sn^n\setminus \{p\}$ is conformally equivalent to the Euclidean space $\En^n$. Every conformal map on $\En^n$ extends to a unique conformal map on $\Sn^n$. Therefore, classifying the cohomogeneity one (resp. zero) conformal actions on the Euclidean space $\En^n$ is a result of the classification of the cohomogeneity one (resp. zero) conformal actions on the sphere $\Sn^n$.

%\begin{theorem}\label{thm.Li}(Liouville's Theorem)
%Let $U,V\subset \En^n$ ($n\geq 3$) be two connected nonempty open subsets, and $f:U\rightarrow V$ be a conformal map. Then $f$ extends to a global conformal map on $\En^n$. 
%\end{theorem}

%Let $(M,\gi)$ be a pseudo-Riemannian manifold, and $G,H$ Lie subgroups of $\Conf(M,\gi)$. The actions of $G$ and $H$ on $M$ are called \textbf{orbit equivalent} if there exists a conformal map $\varphi:M\rightarrow M$ such that for all $p\in M$ the orbit $H(\varphi(p))$ coincides with $\varphi\big(G(p)\big)$. Observe that, If $G$ and $H$ are conjugate within $\Conf(M,\gi)$, then their actions are orbit equivalent. It is easy to see that the reverse implication is not necessarily true. However, the following lemma gives a powerful tool to distinguish the orbit equivalent actions.
%\begin{lemma}
%Let $G,H\subset \Conf(M,\gi)$ be connected Lie subgroups, and $\varphi$ be a conformal map on $M$. Then the following statements are equivalent:
%\begin{itemize}
%\item $G$ and $H$ are orbit equivalent via $\varphi$;
%\item for all $p\in M$,
%\[d\varphi_p(T_pG(p))=T_{\varphi(p)}(H(\varphi(p))).\]
%\end{itemize}
%\end{lemma}

In this paper, we classify, up to conjugacy, the cohomogeneity one and cohomogeneity zero conformal actions on the $3$-dimensional Riemannian essential spaces, namely, the sphere $\Sn^3$, and the Euclidean space $\En^3$. Actually, subgroups of $\Conf(\Sn^3)$ fixing a point in $\Sn^3$ can be seen as the subgroups of $\Conf(\En^3)$ (see Section \ref{sec.euc}).

Let $\R^{1,4}$ denote the $5$-dimensional real vector space endowed with the Lorentzian quadratic form 
$$\mathfrak{q}(v)=-v_1^2+v_2^2+v_3^2+v_4^2+v_{5}^2.$$
The image of the nullcone
$$\Nu^{1,4}=\{v\in \R^{1,4}\setminus \{0\}:\mathfrak{q}(v)=0\}$$
under the projection $\Pn:\R^{1,4}\rightarrow\R\Pn^{4}$ is a submanifold diffeomorphic to the $3$-dimensional sphere $\Sn^3$. The degenerate metric on $\Nu^{1,4}$ induces a Riemannian conformal structure $[ds^2]$ on $\Sn^{3}\subset \R\Pn^{4}$. The structure $(\Sn^3,[ds^2])$ is conformally equivalent to the natural round Riemannian metric $\gi_{\Sn}$ on $\Sn^{3}$ of constant sectional curvature $1$.

The action of the Lorentz group $O(1,4)$ on $\R^{1,4}$ leaves the nullcone invariant. This induces a natural conformal action of $O(1,4)$ on the sphere $\Sn^3$. Actually, $O(1,4)$ is the full conformal group of $\Sn^3$. The identity component of $O(1,4)$, denoted by $SO_\circ(1,4)$, consists of the linear isometries of $\R^{1,4}$ preserving both orientation and time-orientation.

The following theorems classify connected Lie subgroups of $\Conf(\Sn^3)$ acting with cohomogeneity one and cohomogeneity zero on $\Sn^3$.
\begin{theorem}\label{thm.m1}
Let $G\subset \Conf(\Sn^3)$ be a connected Lie subgroup which acts on $\Sn^3$ with cohomogeneity one. Then, up to conjugacy, $G$ admits the same orbits as those induced by one of the groups in Table \ref{table1}.  Furthermore, the action of $G$ is proper if and only if it is conjugate to $SO(3)$ or $SO(2)\times SO(2)$.
\begin{table}[h!]
\centering
\begin{tabular}{|c| c| c|c| } 
 \hline
$\mathscr{P}$ & $SO(2)\times \el$  & $ \mathcal{N}_a\ltimes \el$ & $\R_+^*\ltimes \el$\\
\hline
 $\R_+^*\times SO(2)$& $SO(3)$  & $SO(2)\times SO(2)$ & $SO_\circ(1,2)$ \\
\hline
\end{tabular}
\caption{Here $a\in \R^*$ is a constant number, $\mathscr{P}=\R e_1\oplus \R e_2$, and $\el=\R e_3$.}
\label{table1}
\end{table}
\end{theorem}
\begin{proof}
It follows from Lemma \ref{lem.sub}, Proposition \ref{pro}, and Theorems \ref{thm.euc.1}, \ref{thm.p.1}.
\end{proof}
\begin{theorem}\label{thm.m2}
Let $G\subset \Conf(\Sn^3)$ be a connected Lie subgroup which acts on $\Sn^3$ with cohomogeneity zero. Then, up to conjugacy, $G$ admits the same orbits as those induced by one of the groups in Table \ref{table2}. Furthermore, if $G$ acts on $\Sn^3$ properly, then the action is transitive.
\begin{table}[h!]
\centering
\begin{tabular}{|c| c| c|c| } 
 % \multicolumn{4}{|c|}{\textbf{Subgroups acting on $\En^3$ by cohomogeneity zero or one}} \\
  \hline
 $\R^{3}$ & $\R_+^*\ltimes \mathscr{P}$  & $(\R_+\times SO(2))\ltimes \el$ & $\R_+^*\times SO(3)$ \\
  \hline
$SO_\circ(1,3)$ & $SO_\circ(1,2)\times SO(2)$ & $SO_\circ(1,4)$ & \\
\hline
\end{tabular}
\caption{Here $\mathscr{P}=\R e_1\oplus \R e_2$, and $\el=\R e_3$.}
\label{table2}
\end{table}
\end{theorem}
\begin{proof}
It follows from Lemmas \ref{lem.sub}, \ref{lem.prop}, and Theorems \ref{thm.euc.0}, \ref{thm.p.0}.
\end{proof}

The following theorem plays a key role in this study.
\begin{theorem}\cite{Di1}\label{thm.1}
Let $G$ be a connected Lie subgroup of $O(1,n)$ which acts on $\R^{1,n}$ irreducibly. Then $G=SO_\circ(1,n)$.
\end{theorem}

Thanks to the above theorem, every connected proper Lie subgroup of $SO_\circ(1,4)$ preserves a non-trivial linear subspace in $\R^{1,4}$. The following lemma divides the study of conformal actions on the three dimensional essential Riemannian manifolds into the following three sections.

\begin{lemma}\label{lem.sub}
Let $G\subset \Conf_\circ(\Sn^3)=SO_\circ(1,4)=\Iso_\circ(\Hn^4)$ be a non-trivial proper connected Lie subgroup. Then $G$ fixes a point in $\Sn^3$ or in $\Hn^4$ or it preserves a positive definite linear subspace  of $\R^{1,4}$ of dimension $1$ or $2$.
\end{lemma}
\begin{proof}
By Theorem \ref{thm.1}, $G$ preserves a non-trivial linear subspace $V\leq \R^{1,4}$, since it is a proper subgroup. The group $G$ also preserves the orthogonal space $V^\perp$. Since $\dim V^\perp=5-\dim V$, we can restrict ourselves to $\dim V\leq 2$. If $V$ is a lightlike (resp. timelike) line, then $G$ fixes a point in $\Sn^3$ (resp. $\Hn^4$). Suppose that $V$ is a Lorentzian $2$-plane. Then $V$ contains exactly two distinct lightlike lines. Since $G$ is connected and its action on $V$ is isometric, it preserves both the lines. If $V$ is a degenerate $2$-plane, then it contains a unique lightlike line. This line is $G$-invariant, since the action of $G$ is isometric.
\end{proof}

We need the following well-known lemma in squeal.
\begin{lemma}\label{lem.so}
The orthogonal group $O(3)$ has three distinct connected Lie subgroups, up to conjugacy, namely, $\{Id\}$, $SO(2)$, and $SO(3)$.
\end{lemma}
\section{Proper actions}\label{sec.proper}
In this section we consider the proper actions on $\Sn^3$. Let $G$ be a connected Lie subgroup of $\Conf(\Sn^3)$ which acts on $\Sn^3$ properly. Then $G$ must be a compact Lie group since $\Sn^3$ itself is compact. Therefore, $G$ is, up to conjugacy, a subgroup of $SO(4)$, the maximal compact Lie subgroup of $SO_\circ(1,4)$. Actually, $SO(4)$ is the identity component of the stabilizer of a timelike line $\ell\leq \R^{1,4}$. Hence, it acts on $\ell$ trivially, and so, fixes a point in the hyperbolic space $\Hn^4$. The group $SO(4)$ preserves also the orthogonal complement $\ell^\perp\leq \R^{1,4}$ which is isometric to the Euclidean space $\R^4$. Hence, we may consider the action of $G$ on $\Sn^3$ as an action induced from its linear isometric action on $\R^4$, where $\Sn^3$ is the set of unit vectors in $\R^4$. It follows that $G$ acts on $(\Sn^3,\gi_{\Sn})$ isometrically.
\begin{lemma}\label{lem.prop}
Let $G\subset \Conf(\Sn^3)$ be a connected Lie subgroup which acts on $\Sn^3$ properly and with cohomogeneity zero. Then the action is transitive.
\end{lemma}
\begin{proof}
Observe that, $G$ is compact, since the action is proper and $\Sn^3$ is compact. Hence, every $G$-orbit in $\Sn^3$ is closed. The group $G$ admits an open orbit, say $\mathcal{O}\subset \Sn^3$. Therefore, $\mathcal{O}$ coincides with $\Sn^3$, since it is a non-empty open and closed subset of a connected space. Hence, $G$ acts on $\Sn^3$ transitively.
\end{proof}
\begin{proposition}\label{pro}
Let $G\subset \Conf(\Sn^3)$ be a connected Lie subgroup which acts on $\Sn^3$ properly and with cohomogeneity one. Then it is conjugate either to $SO(3)$ or $SO(2)\times SO(2)$.
\end{proposition}
\begin{proof}
Since $G$ acts properly and $\Sn^3$ is compact, $G$ is a compact subgroup of $SO(4)$ up to conjugacy. By a well known fact in cohomogeneity one literature, the orbit space may be homeomorphic to one of the following spaces (see \cite{Mos}).
$$[0,1]\quad ,\quad [0,1)\quad ,\quad \mathbb{S}^1\quad ,\quad \mathbb{R}.$$
By the compactness of $\Sn^3$ the cases $[0,1)$ and $\R$ are excluded. We claim that the case $\Sn^1$ is discarded too. If the orbit space $\Sn^3/G$ is homeomorphic to $\Sn^1$ then the projection $\pi:\Sn^3\rightarrow \Sn^1$ is a fibration with fiber $G/K$, where $K$ is the isotropy subgroup of a regular point (regular point is a point that its image in the orbit space is an internal point). By Theorem 4.41 of \cite[p.379]{Hat} there is a long exact sequence of homotopy groups as follows.
$$  \rightarrow \pi_m(G/K) \rightarrow \pi_m(\Sn^{3}) \rightarrow \pi_m(\Sn^1
) \rightarrow \pi_{m-1}(G/K) \rightarrow \cdots \rightarrow \pi_0(\Sn^3
) \rightarrow 0, $$
which implies that $\pi_1(\Sn^1)=0$, a contradiction. Thus the orbit space is homeomorphic to $[0,1]$, and so there exist two singular orbits. Non of the singular orbits is exceptional, since $\Sn^3$ is simply connected (see \cite[p.185]{Br}). Thus each singular orbit is either a point or a circle. 

Let $G(x)$ be a singular orbit. 

%First we show that the action of $G$ on $\R^4$ is reducible. Let $x\in \Sn^3$ be a point such that the orbit $G(x)$ is singular. 
\textbf{Case I:} If $G(x)$ is a point (the orbit is zero-dimensional) then clearly $G$ acts on $\R^4$ reducibly. Precisely, $G$ preserves the line $\R x\leq \R^4$. Hence, up to conjugacy, it is a subgroup of $SO(3)$. In the one hand, $SO(3)$ fixes two antipodal points on $\Sn^3$ and admits a codimension one foliation on the complement of these two points in which every leaf is homothetic to the $2$-sphere $\Sn^2$. On the other hand, by Lemma \ref{lem.so}, $SO(3)$ has no $2$-dimensional Lie subgroup. This implies that $G=SO(3)$, up to conjugacy.

\textbf{Case II:} Now, assume that $G(x)$ is a $1$-dimensional orbit. The action of $G$ on $G(x)$ induces the representation $\pi:G\rightarrow \Iso(G(x))$. Let $K$ denote the stabilizer $G_x$. Obviously, $\dim K\geq 1$, since $\dim G\geq 2$. The subgroup $K$ is closed, and so, compact. Hence, for an arbitrary point $y\in G(x)$ the orbit $K(y)$ is compact. It follows that $K$ acts on $G(x)$ trivially. We deduce that for any $y\in G(x)$ the stabilizer $G_y$ coincides with $K$. Consequently, $K$ is the kernel $\ker \pi$. So, $K$ is a normal subgroup of $G$. Since, $K$ fixes at least two non-antipodal points in $\Sn^3$, it preserves a $2$-plane $\Pi\leq \R^4$. In fact, $K$ acts on $\Pi$ trivially, since its action is isometric. Also, $K$ preserves the orthogonal complement $\Pi^\perp$ which is of dimension $2$. Note that, since the action of $K$ on $\R^4$ is faithful, it admits an orbit of dimension $\geq 1$. Therefore, it acts on $\Pi^\perp$ irreducibly. So, $\Pi$ is the maximal subset of $\R^4$ which $K$ acts on it trivially. This induces a faithful representation of $K$ in the group of linear isometries of $\Pi^{\perp}$ (the kernel of this representation will be in the kernel of the action of $G$ on $\Sn^3$). Hence $K\subseteq O(2)$ and so $\dim G=2$. 
For any $g\in G$ and any $v\in \R^4$ we have $g(K(v))=K(g(v))$, since $K$ is a normal subgroup of $G$. It follows that, for an arbitrary point $w\in \Pi$ the orbit $K(g(w))$ is a singleton. This implies that the element $g$ preserves $\Pi$, and so, $\Pi$ is $G$-invariant. This induces a representation $G\rightarrow SO(2)\times SO(2)$, where $SO(2)$ is the identity component of the group of linear isometries of $\Pi$ (and $\Pi^\perp$). This map is onto, since the action is of cohomogeneity one. Since $\dim G=2$, the kernel of this representation is discrete, and by the faithfulness of the action it should be trivial.

\end{proof}

	By the proof of Proposition \ref{pro} one gets that the action of $SO(2)\times SO(2)$ admits two $1$-dimensional orbits $\mathcal{O}_1,\mathcal{O}_2$ each of them isometric to the circle $\Sn^1$, the one consists of the unit vectors in $\Pi$, and the other corresponds to those of $\Pi^\perp$. Furthermore, $SO(2)\times SO(2)$ acts on the complement $\Sn^3\setminus(\mathcal{O}_1\cup\mathcal{O}_2)$ freely. Hence, it admits a codimension 1 foliation on $\Sn^3\setminus (\mathcal{O}_1\cup\mathcal{O}_2)$ on which every leaf is homothetic to the $2$-torus $\mathbb{T}^2$. 

As an immediate consequence of the proof of Proposition \ref{pro} one gets the following corollary.

\begin{corollary}
Let $G\subset SO(4)$ be a connected Lie subgroup which acts on $\Sn^3$ with cohomogeneity one. Then, $G$ is a closed subgroup of $SO(4)$, therefore, it acts on $\Sn^3$ properly.
\end{corollary}
\begin{proof}
Consider the action of $G$ on $\R^4$. If $G$ acts on $\R^4$ reducibly, then by a similar argument of that of the proof of Proposition \ref{pro} one can see that, $G$  is conjugate either to $SO(3)$ or $SO(2)\times SO(2)$, which are closed subgroups of $SO(4)$. If $G$ acts on $\R^4$ irreducibly, then by \cite[Theorem 1]{Di2}, $G$ is a closed subgroup of $GL(4,\R)$. This implies that $G$ is a closed subgroup of $SO(4)$. This completes the proof.
\end{proof}
\section{Actions on the $3$-dimensional Euclidean space $\En^3$}\label{sec.euc}
Let $G$ be a Lie subgroup of $\Conf(\Sn^3)$ which fixes a point $p_0\in \Sn^3$. As we mentioned earlier, the complement $\Sn^3\setminus \{p_0\}$ is conformally equivalent to the $3$-dimensional Euclidean space $\En^3$. This induces a representation form $G$ into $\Conf(\En^3)$. Choosing a point $o\in \En^{3}$ as the \textit{origin}, the conformal group $\Conf(\En^{3})$ splits as the semi-direct product $(\R^*\times O(3))\ltimes \R^3$. Thus, $\Conf(\En^3)$ consists of the elements of the following form
\[f_{\alpha,A,v}:\En^3\rightarrow \En^3,\;\;\;\;\;\;\;x\mapsto \alpha A(x)+v,\]
where $\alpha\neq 0$ is a homothety (centred at $o$), $A\in O(3)$ a linear isometry, and $v\in \R^3$ a translation. 
%Therefore, $\Conf(\En^3)$ is isomorphic to the semi-direct product $(\R^*\times O(3))\ltimes \R^3$. 

The group operation rule in $\Conf(\En^3)\simeq (\R^*\times O(3))\ltimes \R^3$ is as follows:
\[(\alpha,A,v).(\beta,B,w)=:(\alpha\beta,AB,\alpha A(w)+v),\;\;\;\;\;\;\;\alpha,\beta\in \R^*,\;A,B\in O(3),\;v,w\in \R^3.\]

The Lie algebra of $\Conf(\En^3)$ is isomorphic to the semi-direct sum $(\R\oplus \so(3))\oplus_\pi\R^3$, where $\pi$ is the natural representation of $\R\oplus \so(3)$ in $gl(\R^3)$. Hence, the Lie bracket rule in this Lie algebra is 
\[[\eta+V+v,\xi+W+w]=[V,W]+V(w)+\eta w-W(v)-\xi v,\;\;\;\;\;\;\;\eta,\xi\in \R,\;V,W\in \so(3),\;v,w\in \R^3.\]

Therefore, the adjoint action of $\Conf(\En^3)$ on its own Lie algebra is as follows. For $(r,A,v)\in (\R^*\times O(3))\ltimes \R^{3}$
\[Ad_{(r,A,v)}:(\R\oplus \so(3))\oplus_\pi \R^3 \longrightarrow (\R\oplus \so(3))\oplus_\pi \R^3,\;\;\;\;\;\;\;  b+W+w\mapsto b+AWA^{-1}+rA(w)-bv-AWA^{-1}(v).\]

The following maps are Lie algebra morphisms:
\begin{align*}
&p_l:(\R\oplus \so(3))\oplus_\pi\R^3\longrightarrow \R\oplus \so(3),\;\;\;\;\;a+V+v\mapsto a+V\\
&p_{li}:(\R\oplus \so(3))\oplus_\pi\R^3\longrightarrow \so(3),\;\;\;\;\;\;a+V+v\mapsto V\\
&p_h:(\R\oplus \so(3))\oplus_\pi\R^3\longrightarrow \R,\;\;\;\;\;\;\;a+V+v\mapsto a.
\end{align*}

For a Lie subalgebra $\g$ of $(\R\oplus \so(3))\ltimes \R^3$, the kernel $\ker p_l|_\g$ is called the translation part of $\g$ and we denote it by $T(\g)$. It is a linear subspace of $\R^3$ and an ideal of $\g$, and so, it is invariant by the natural action of $p_{li}(\g)$.

Let $\{e_1,e_2,e_3\}$ be the standard basis for $\R^3$. The set of the following matrices is a basis for $\so(3)$:
\begin{align*}
X=\begin{bmatrix}
0 & 1 & 0\\
-1 & 0 & 0\\
0 & 0 & 0
\end{bmatrix},\;\;\;\;Y=\begin{bmatrix}
0 & 0 & 1\\
0 & 0 & 0\\
-1& 0 & 0
\end{bmatrix},\;\;\;\;\;Z=\begin{bmatrix}
0 & 0 & 0\\
0 & 0 & 1\\
0 & -1 & 0
\end{bmatrix}.
\end{align*}
We have $[X,Y]=-Z$, $[X,Z]=Y$, and $[Y,Z]=-X$. Also, we have $X(e_1)=-Z(e_3)=-e_2$, $X(e_2)=Y(e_3)=e_1$, $Y(e_1)=Z(e_3)=-e_3$, and $X(e_3)=Y(e_2)=Z(e_1)=0$.

By Lemma \ref{lem.so}, $\so(3)$ has a unique non-trivial proper Lie subalgebra, up to conjugacy, namely, the subalgebra generated by $X$. The $X$-invariant subspace of $\R^3$ are the $2$-plane $\mathscr{P}$ generated by $\{e_1,e_2\}$ and the line $\el$ generated by $e_3$.
 
 We specify the following $1$-parameter subgroups of $(\R^*\times O(3))\ltimes \R^3$
 \begin{align*}
 &\mathcal{N}_a=\left\lbrace\exp\big(t(a+X)\big):t\in \R\right\rbrace=\left\lbrace\left(e^{at},\begin{bmatrix}
 \cos t & \sin t & 0\\
 -\sin t & \cos t & 0\\
 0 & 0 & 1
 \end{bmatrix},\begin{bmatrix}
 0\\
 0\\
 0
 \end{bmatrix}\right):t\in \R\right\rbrace,\\
& \mathcal{S}=\left\lbrace\exp\big(t(X+e_3)\big):t\in\R\right\rbrace=\left\lbrace\left(1,\begin{bmatrix}
 \cos t & \sin t & 0	\\
 -\sin t & \cos t & 0\\
 0 & 0 & 1
 \end{bmatrix},\begin{bmatrix}
 0\\
 0\\
 t
 \end{bmatrix}\right):t\in \R\right\rbrace,
 \end{align*}
 where $a\in \R^*$ is a constant number.
\begin{theorem}\label{thm.euc.1}
Let $G\subset \Conf(\En^3)$ be a connected Lie subgroup which acts on $\En^3$ with cohomogeneity one. Then, up to conjugacy, $G$ admits the same orbits as those induced by one of the following groups
\begin{align}\label{list.euc.1}
\mathscr{P},\;\;\;\;\;\;  SO(2)\times \el,\;\;\;\;\;\;\mathcal{N}_a\ltimes \el,\;\;\;\;\;\;\R_+^*\ltimes \el,\;\;\;\;\;\;SO(3),\;\;\;\;\;\;\R_+^*\times SO(2).
\end{align}
\end{theorem}

\begin{theorem}\label{thm.euc.0}
Let $G\subset\Conf(\En^3)$ be a connected Lie subgroup which acts on $\En^3$ with cohomogeneity zero. Then, up to conjugacy, $G$ admits the same orbits as those induced by one of the following groups
\begin{align}\label{list.euc.0}
\R^3,\;\;\;\;\;\;\; \R_+^*\ltimes \mathscr{P},\;\;\;\;\;\;\;(\R_+^*\times SO(2))\ltimes \el,\;\;\;\;\;\;\;\R_+^*\times SO(3).
\end{align}
\end{theorem}

Theorems \ref{thm.euc.1} and \ref{thm.euc.0} follow from the following proposition. It classifies, up to conjugacy, all the connected Lie subgroups of $(\R^*\times O(3))\ltimes \R^3$ with $\dim \geq 2$. 

\begin{proposition}\label{prop.group}
Let $G$ be a connected Lie subgroup of $(\R^*\times O(3))\ltimes \R^3$ with $\dim G\geq 2$. Then $G$ is conjugate to one of the Lie subgroups in Table \ref{table5}.
\end{proposition}
\begin{table}[h!]
\centering
\begin{tabular}{|c| c| c|c| } 
% \hline
%  \multicolumn{4}{|c|}{\textbf{Subgroups acting on $\En^3$ by cohomogeneity zero or one}} \\
  \hline
 $\R^*_+\ltimes \R^{3}$ & $(\R_+^*\times SO(3))\ltimes \R^{3}$ & $SO(3)\ltimes \R^3$ &  $\mathcal{N}_a \ltimes \R^{3}$\\
  \hline
 $\R^{3}$ & $SO(2)\ltimes \R^3$ & $(\R_+^*\times SO(2))\ltimes \R^3$ & $\mathcal{S}\ltimes \mathscr{P}$\\
  \hline
$SO(2)\ltimes \mathscr{P}$  &  $(\R_+^*\times SO(2))\ltimes \mathscr{P}$ &  $\R_+^*\ltimes \mathscr{P}$  & $\mathcal{N}_a\ltimes \mathscr{P}$\\
  \hline
$\mathscr{P}$ &  $SO(2)\times \el$ & $(\R_+^*\times SO(2))\ltimes \el$ & $ \mathcal{N}_a\ltimes \el$ \\

  \hline
$\R_+^*\ltimes \el$ &  $\R_+^*\times SO(3)$ & $SO(3)$& $\R_+^*\times SO(2)$ \\
\hline
\end{tabular}
\caption{Here $a\in \R^*$ is a constant number, $\mathscr{P}=\R e_1\oplus \R e_2$, and $\el=\R e_3$.}
\label{table5}
\end{table}
\begin{proof}
We classify the connected Lie subgroups of $H=(\R^*\times O(3))\ltimes \R^3$ by their corresponding Lie subalgebras in $\h=(\R\oplus \so(3))\oplus_\pi\R^3$ and applying the adjoint action of $H$ on $\h$. We take $\lambda=1$ as a generator (basis) of $\R$.

Let $\g$ be a Lie subalgebra of $\h$ with $\dim \g\geq 2$.

First assume that $\dim p_{li}(\g)=3$. Then clearly $p_{li}(\g)=so(3)$. The $\so(3)$-invariant subspace of $\R^3$ are $\R^3$ itself and $\{0\}$. Therefore, the translation part of $T(\g)$ is either $\R^3$ or $\{0\}$
\begin{itemize}
\item If $T(\g)=\R^3$:
\begin{itemize}
\item If $\dim p_l(\g)=4$, then obviously $\g=\h$.
\item If $\dim p_l(\g)=3$, it follows that $p_{li}:\g\rightarrow \so(3)$ is a Lie algebra isomorphism. Hence $f=p_h\circ p_{li}^{-1}:\so(3)\rightarrow p_h(\g)$ is a surjective Lie algebra morphism. Since $\so(3)$ is simple, it has no non-trivial ideal. It follows that $f\equiv 0$. This implies that $\g$ is conjugate to $\so(3)\oplus_\pi\R^3$.
\end{itemize}
\item If $T(\g)=\{0\}$:
\begin{itemize}
\item If $\dim p_l(\g)=4$, then there are four vectors $u,v,w,s\in \R^3$ such that $\{\lambda+u,X+v,Y+w,Z+s\}$ is a basis for $\g$. Considering the Lie bracket we have
\begin{align*}
[\lambda+u,X+v]= v-X(u)=0,\;\;\;\;\;\;\;[\lambda+u,Y+w]=w-Y(u)=0,\;\;\;\;\;[\lambda+u,Z+s]=s-Z(u)=0.
\end{align*}
From the above equations, we get
\begin{align*}
&Ad_{(1,Id,u)}(\lambda+u)=\lambda+u-u=\lambda,\;\;\;\;\;\;\;\;\;\;\;\;\;\;\;\;\;Ad_{(1,Id,u)}(X+v)=X+v-X(u)=X,\\
&Ad_{(1,Id,u)}(Y+w)=Y+w-Y(u)=Y,\;\;\;\;\;\;\;\;\;\;Ad_{(1,Id,u)}(Z+s)=Z+s-Z(u)=Z,
\end{align*}
Therefore, $\g$ is conjugate to $\R\oplus \so(3)$.
\item If $\dim p_l(\g)=3$, then $p_{li}(\g)=\so(3)$. It follows that $p_{li}:\g\rightarrow \so(3)$ is an isomorphism. Hence $\g$ is simple. Thus it is conjugate to the Levi factor of $\h$ which is $\so(3)$.
\end{itemize}
\end{itemize}
Now, suppose that $p_{li}(\g)$ is a $1$-dimensional Lie subalgebra. Since $\so(3)$ has, up to conjugacy, a unique $1$-dimensional Lie subalgebra, we may assume $p_{li}(\g)=\R X$. The $X$-invariant linear subspace of $\R^3$ are $\{0\}$, the line $\el=\R e_3$, the $2$-plane $\mathscr{P}=\R e_1\oplus \R e_2$, and $\R^3$.
\begin{itemize}
\item If $T(\g)=\R^3$:
\begin{itemize}
\item If $\dim p_l(\g)=2$, then clearly $\g=(\R\oplus \R X)\oplus_\pi\R^3$.
\item If $\dim p_{l}(\g)=1$, then there is a constant $a\in \R$ such that $\{a+X,e_1,e_2,e_3\}$ is a basis for $\g$. Therefore, $\g=\R(a+X)\oplus_\pi\R^3$.
\end{itemize}
\item If $T(\g)=\mathscr{P}$:
\begin{itemize}
\item If $\dim p_l(\g)=2$, then there are two vectors $u,v\in \R^3$ such that $\{\lambda+u,X+v,e_1,e_2\}$ is a basis for $\g$. By closeness of the Lie bracket in $\g$ we have
\[[\lambda+u,X+v]=v-X(u)\in \mathscr{P}.\]
Now,
\[Ad_{1,Id,u}(\lambda+u)=\lambda+u-u=\lambda,\;\;\;\;\;\;Ad_{(1,Id,u)}(X,v)=X+v-X(u).\]
Since $v-X(u)\in \mathscr{P}$, we have $X\in Ad_{(1,Id,u)}(\g)$. Hence, $\g$ is conjugate to the semi-direct sum $(\R\oplus\R X)\oplus_\pi\mathscr{P}$.
\item If $\dim p_l(\g)=1$, then there is a constant $a\in \R$ and a vector $u\in \R^3$ such that $\{a+X+u,e_1,e_2\}$ is a basis for $\g$. For a vector $x\in \R^3$ we have
\[Ad_{(1,Id,x)}(a+X+u)=a+X+u-ax-X(x).\]
If $a\neq 0$, the linear system $u-ax-X(x)=0$ has a unique solution, namely, $x=\left(\dfrac{au_1+u_2}{a^2+1},\dfrac{au_2+u_1}{a^2+1},\dfrac{u_3}{a}\right)$. Hence, $\g$ is conjugate to $\R(a+X)\oplus_\pi\mathscr{P}$.

If $a=0$, then $u-ax-X(x)\in \mathscr{P}$ if and only if $u_3=0$. Therefore, there are two cases: If $u_3=0$, then $\g$ is conjugate to the semi-direct sum $\R X\oplus_\pi\mathscr{P}$. Otherwise, $\g$ is conjugate to $ \R(X+e_3)\oplus_\pi\mathscr{P}$ via $Ad_{(1/u_3,Id,0)}$.
\end{itemize}
\item If $T(\g)=\el=\R e_3$:
\begin{itemize}
\item If $\dim p_l(\g)=2$, then there are two vectors $u,v\in \R^3$ such that $\{\lambda+u,X+v,e_3\}$ is a basis for $\g$. Considering the Lie bracket we have
\[[\lambda+u,X+v]=v-X(u)\in \el.\]
Now, we have 
\[Ad_{(1,Id,u)}(\lambda+u)=\lambda+u-u=\lambda,\;\;\;\;\;\;Ad_{(1,Id,u)}(X+v)=X+v-X(u).\]
It follows that $X\in \g'=Ad_{(1,Id,u)}(\g)$, since $v-X(u)\in T(\g')=T(\g)=\el$. Hence, $\g$ is conjugate to $(\R\oplus\R X)\oplus_\pi \el$.
\item If $\dim p_l(\g)=1$, then there is a constant $a\in \R$ and a vector $u\in \R^3$ such that $\{a+X+u,e_3\}$ is a basis for $\g$. We apply $Ad_{(1,Id,x)}$ on $\g$:
\[Ad_{(1,Id,x)}(a+X+u)=a+X+u-ax-X(x),\;\;\;\;\;\;Ad_{(1,Id,x)}(e_3)=e_3+x-x=e_3.\]
Setting $x=\left(\dfrac{au_1+u_2}{a^2+1},\dfrac{au_2+u_1}{a^2+1},\dfrac{u_3}{a}\right)$ we get $u-ax-X(x)\in \el$. Hence $a+X\in \g'=Ad_{(1,Id,x)}(\g)$, since $u-ax-X(x)\in T(\g')=T(\g)=\el$. Therefore, $\g$ is conjugate to the semi-direct sum $\R(a+X)\oplus_\pi \el$.
\end{itemize}
If $T(\g)=\{0\}$, then there are two vectors $u,v\in \R^3$ such that $\{\lambda+u,X+v\}$ is a basis for $\g$. By closeness of the Lie bracket we have:
\[[\lambda+u,X+v]=v-X(u)=0.\]
Hence, applying $Ad_{(1,Id,u)}$ on $\g$ we have
\[Ad_{(1,Id,u)}(\lambda+u)=\lambda+u-u=\lambda,\;\;\;\;\;Ad_{(i,Id,u)}(X+v)=X+v-X(u)=X.\]
Therefore, $\g$ is conjugate to $\R\oplus \R X$.
\end{itemize}
Finally, If $p_{li}(\g)=\{0_{3\times 3}\}$, then $\g$ is a Lie subalgebra of $\R\oplus_\pi \R^3$. Up to action of $SO(3)$ on $\R^3$, the subspaces $\el$ and $\mathscr{P}$ are unique $1$-dimensional and $2$-dimensional linear subspace of $\R^3$, respectively. Hence, if $p_h(\g)=\{0\}$, then $\g$ is conjugate to $\mathscr{P}$ or $\R^3$. From now on, we assume that $p_h(\g)=\R$.
\begin{itemize}
\item If $T(\g)=\R^3$, then obviously $\g=\R\oplus \R^3$.
\item If $T(\g)=\mathscr{P}$, then there is a vector $u\in \R^3$ such that $\{\lambda+u,e_1,e_2\}$ is a basis of $\g$. Applying $Ad_{(1,Id,u)}(\g)$ we get $Ad_{(1,Id,u)}(\g)=\R\oplus_\pi \mathscr{P}$. Therefore, $\g$ is conjugate to $\R\oplus_\pi\mathscr{P}$.
\item If $T(\g)=\el$,  then there is a vector $u\in \R^3$ such that $\{\lambda+u,e_3\}$ is a basis of $\g$. Applying $Ad_{(1,Id,u)}(\g)$ we get $Ad_{(1,Id,u)}(\g)=\R\oplus_\pi \el$. Therefore, $\g$ is conjugate to $\R\oplus_\pi\el$.
\end{itemize}
\end{proof}

Now, we consider the orbits induced by the actions of the subgroups obtained in Proposition \ref{prop.group}. Let $(x,y,z)$ be an orthonormal coordinate for $\En^{3}$. The origin $o$ has coordinate $(0,0,0)$. 

Consider the action of $G=\mathcal{S}\ltimes \mathscr{P}$ on $\En^3$. For an arbitrary point $p=(x,y,z)\in \En^3$, the vector tangent to the orbit $G(p)$ at $p$ induced by the $1$-parameter subgroup $\mathcal{S}\subset G$ is $v=(y,-x,1)$. Clearly, $\{v,e_1,e_2\}$ is a basis for $T_pG(p)$. Hence, all the orbits induced by $G$ in $\En^3$ are open, and since, $\En^3$ is connected, we conclude that, $G$ acts on $\En^3$ transitively. Also, every subgroup of $\Conf(\En^3)$ containing the full translation group $\R^3$ acts on $\En^3$ transitively. Therefore, the actions of the following groups on $\En^3$ are transitive:
\begin{align*}
& \R^*_+\ltimes \R^{3},\;\;\;\;\; (\R_+^*\times SO(3))\ltimes \R^{3},\;\;\;\;\; SO(3)\ltimes \R^3,\;\;\;\;\;\mathcal{N}_a \ltimes \R^{3},\\
& \R^{3},\;\;\;\;\; SO(2)\ltimes \R^3,\;\;\;\;\;(\R_+^*\times SO(2))\ltimes \R^3,\;\;\;\;\;\mathcal{S}\ltimes \mathscr{P},
\end{align*}

The action of $\mathscr{P}$ admits a codimension $1$ foliation $\mathcal{F}_\mathscr{P}$ on $\En^3$ on which every leaf is an affine $2$-plane. The $1$-parameter subgroup $SO(2)$ preserves the leaves of this foliation. Therefore, the orbits induced by $SO(2)\ltimes \mathscr{P}$ in $\En^3$ are exactly the same as those induced by $\mathscr{P}$.

Let $G=\mathcal{N}_a\ltimes \mathscr{P}$. In the one hand, the homothety factor $\R_+^*$ preserves a unique leaf of $\mathcal{F}_\mathscr{P}$, namely $\mathcal{F}_\mathscr{P}(o)$. This implies that the group $\mathcal{N}_a\ltimes \mathscr{P}\subset (\R_+^*\times SO(2))\ltimes \mathscr{P}$ preserves the leaf $\mathcal{F}_{\mathscr{P}}(o)$, as well. On the other hand, for an arbitrary point $p=(x,y,z)\in \En^3$, the vector tangent to the orbit $G(p)$ at $p$ induced by the $1$-parameter subgroup $\mathcal{N}_a$ is $v=(ax+y,-x+ay,az)$. The set $\{v,e_1,e_2\}\subset T_pG(p)$ is a basis if and only if $z\neq 0$ if and only if $p\notin \mathcal{F}_{\mathscr{P}}(o)$. Hence, $G$ acts on the both connected components of $\En^3\setminus \mathcal{F}_{\mathscr{P}}(o)$ transitively. Therefore, the following groups admit the same orbits in $\En^3$
\[\R^*_+\ltimes \mathscr{P},\;\;\;\;\;\;\;\mathcal{N}_a\ltimes \mathscr{P},\;\;\;\;\;\;\;\;\;(\R_+^*\times SO(2))\ltimes \mathscr{P}.\]

The action of $\el$ admits a $1$-dimensional foliation $\mathcal{F}_\el$ on $\En^3$ on which every leaf is an affine line. The homothety factor $\R^*_+$ and the plane-rotation group $SO(2)$ preserve a unique leaf of $\mathcal{F}_\el$, namely, $\mathcal{F}_\el(o)$. Hence, the leaf $\mathcal{F}_\el(o)$ is invariant by $(\R_+^*\times SO(2))\ltimes \el$.
\begin{itemize}
\item The group $SO(2)\times \el$ acts on $\En^3\setminus \mathcal{F}_\el(o)$, freely. Therefore, it induced a codimension one foliation on $\En^3\setminus \mathcal{F}_\el(o)$ on which every leaf is diffeomorphic to the cylinder $\Sn^1\times \R$.
\item The action of the homothety factor $\R_+^*$ preserves non of the leaves induced by $SO(2)\times \el$ in $\En^{3}\setminus \mathcal{F}_\el(o)$. It follows that $(\R_+^*\times SO(2))\ltimes \el$ acts on $\En^3\setminus \mathcal{F}_{\el}(o)$ transitively.
\item Set $G=\mathcal{N}_a\ltimes \el$. For an arbitrary point $p\in \En^3$, the orbit $G(p)$ intersects the line $l=\{(t,0,0):t\in \R\}$. It can be easily seen that $G$ acts on $\En^3\setminus \mathcal{F}_\el(o)$ freely. Thus every $G$-orbit in $\En^3\setminus \mathcal{F}_\el(o)$ is diffeomorphic to $\R^2$.
\item Every orbit induced by $G=\R_+^*\ltimes \el$ in $\En^3\setminus \mathcal{F}_\el(o)$ intersects the curve $C=\{(\cos\theta,\sin\theta,0):\theta\in \R\}$ in a unique point. Now, one can see $G$ acts on $\En^3\setminus \mathcal{F}_\el(o)$ freely. Actually, every $G$-orbit in $\En^3\setminus \mathcal{F}_\el(o)$ is an affine half-plane.
\end{itemize}

The group $SO(3)$ fixes the origin $o$ and admits a codimension 1 foliation $\mathcal{F}_\Sn$ on $\En^3\setminus \{o\}$ on which every leaf is a $2$-sphere $\Sn^2(r)$ of radius $r>0$. 

The homothety factor $\R_+^*$ preserves no leaf of $\mathcal{F}_\Sn$. Hence, $\R_+^*\times SO(3)$ acts on $\En^3\setminus \{o\}$ transitively.

Finally, the group $G=\R_+^*\times SO(2)$ preserves the line $\mathcal{F}_\el(o)$ and acts on the both connected components of $\mathcal{F}_\el(o)\setminus \{o\}$ transitively. Every $G$ orbit in $\En^3\setminus \mathcal{F}_\el(o)$ intersects the line $l=\{(0,1,t):t\in \R\}$ in a unique point. Now, it can be seen that $G$ acts on $\En^3\setminus \mathcal{F}_\el(o)$ freely. Hence, every $G$-orbit in $\En^3\setminus \mathcal{F}_\el(o)$ is diffeomorphic to $\Sn^1\times \R$.

\textbf{Proofs of Theorems \ref{thm.euc.1} and \ref{thm.euc.0}:} By Proposition \ref{prop.group}, every connected Lie subgroup $G$ of $\Conf(\En^3)$ with $\dim G\geq 2$ is conjugate to one of the groups listed in Table \ref{table5}. The above consideration shows that, if $G$ acts on $\En^{3}$ with cohomogeneity one (resp. zero), then, the $G$-orbits in $\En^3$ are exactly the same as the orbits induced by one of the groups indicated in \eqref{list.euc.1} (resp. \eqref{list.euc.0}).\hfill$\square$
\section{Actions preserving a great sphere in $\Sn^3$}\label{sec.sphe}
The actions admitting a fixed point in the sphere $\Sn^3$ are studied in Section \ref{sec.euc}. Also, if a Lie subgroup $G\subset \Conf_\circ(\Sn^3)=SO_\circ(1,4)=\Iso_\circ(\Hn^4)$ fixes a point in the hyperbolic space $\Hn^4$, then up to conjugacy, it is a subgroup of $SO(4)$. Actions induced by the subgroups of $SO(4)$ are considered in Section \ref{sec.proper}. In this section, we study the actions preserving a positive definite linear subspace $V$ of $\R^{1,4}$. 
%Let $G$ be a connected Lie subgroup of $SO_\circ(1,4)$ which preserves $V$. Then $G$ preserves the orthogonal complement $V^\perp\leq \R^{1,4}$ as well.

\begin{theorem}\label{thm.p.1}
Let $G$ be a connected Lie subgroup of $SO_\circ(1,4)$ which acts on $\Sn^3$ with cohomogeneity one. Furthermore, assume that $G$ fixes no point neither in $\Hn^{4}$ nor in $\Sn^3$. Then $G$ is conjugate to $SO_\circ(1,2)$.
\end{theorem}

\begin{theorem}\label{thm.p.0}
Let $G$ be connected Lie subgroup of $SO_\circ(1,4)$ which acts on $\Sn^3$ with cohomogeneity zero. Furthermore, assume that $G$ fixes no point neither in $\Hn^{4}$ nor in $\Sn^3$. Then $G=SO_\circ(1,4)$ or it is conjugate to either  $SO_\circ(1,3)$ or $SO_\circ(1,2)\times SO(2)$.
\end{theorem}

Let $G$ be a connected Lie subgroup of $\Conf(\Sn^3)=O(1,4)$ which preserves a positive definite linear subspace $V\leq \R^{1,4}$. The group $G$ preserves the orthogonal complement $V^\perp\leq \R^{1,4}$ as well, which is a Lorentzian subspace. There are two cases:

%$1$-dimensional spacelike linear subspace $V$ of $\R^{1,4}$. It can be seen easily that $G$ acts on $V$ trivially. The group $G$ preserves also, 
\begin{itemize}
\item[-] If $V$ is $1$-dimensional, then $G$ acts on $V$ trivially, and so, it is a subgroup of $SO_\circ(1,3)$, up to conjugacy. The orthogonal complement $V^\perp$ is a Lorentzian space of signature $(1,3)$. The projection by $\Pn$ of the nullcone $\Nu^{1,3}$ of $V^\perp\approx\R^{1,3}$ is a great $2$-sphere $\Sn^2$. We show that the complement of $\Sn^2$ in $\Sn^3$ has two connected components each of them conformally equivalent to the $3$-dimensional hyperbolic space $\Hn^3$. Let $\{e_1,e_2,e_3,e_4\}$ be an orthonormal basis for $V^\perp$ which $e_1$ is a timelike vector, and let $e_5\in V$ be a unit vector. The set $B=\{e_1,\cdots,e_5\}$ is an orthonormal basis of $\R^{1,4}$. For an arbitrary point $p\in \Nu^{1,4}\subset \R^{1,4}$ we denote its coordinate respect to the basis $B$ with $(t,x,y,z,u)$. We have 
\begin{align}\label{eq.1}
p\in \Nu^{1,4}\Longrightarrow -t^2+x^2+y^2+z^2=-u^2.
\end{align}
 If $u=0$, then $p\in \Nu^{1,3}\subset V^\perp\approx\R^{1,3}$, and so $\Pn(p)\in \Sn^2$. If $u\neq 0$, dividing $p$ by $|u|$ we may assume 
 \[-t^2+x^2+y^2+z^2=-1.\]
 which describes the $3$-dimensional hyperbolic space $\Hn^3$. One component corresponds to $u<0$ and other to $u>0$. It can be easily seen that this identification is $SO_\circ(1,3)$-equivariant with respect to the natural action of $SO_\circ(1,3)$ on $\Hn^3$. Thus, $SO_\circ(1,3)$ acts on $\Sn^2$ and the both connected components of $\Sn^3\setminus \Sn^2$ transitively.
 \item[-] If $V$ has dimension $2$, then its orthogonal complement $V^\perp \leq \R^{1,4}$ is a $3$-dimensional Lorentzian space. This induces a representation from $G$ into $SO_\circ(1,2)\times SO(2)$. The projection by $\Pn$ of the nullcone $\Nu^{1,2}$ of $V^\perp\approx\R^{1,2}$ is a great circle $\Sn^1$. We show that the complement of $\Sn^1$ in $\Sn^3$ is conformally equivalent to the direct product $\Hn^2\times \Sn^1$. Let $\{e_1,e_2,e_3\}$ be an orthonormal basis of $V^\perp$ which $e_1$ is a timelike vector. Also, assume that $\{e_4,e_5\}$ is an orthonormal basis of $V$. The set $B=\{e_1,\cdots,e_5\}$ is an orthonormal basis of $\R^{1,4}$. For an arbitrary point $p=(t,x,y,z,u)\in \Nu^{1,4}$ in coordinate $B$, we have 
\[-t^2+x^2+y^2=-z^2-u^2.\]
If $p\in V^{\perp}$ (i.e. $z=u=0$), $\Pn(p)$ belongs to $\Pn(\Nu^{1,2})=\Sn^1\subset \Sn^3$.
% The group $SO_\circ(1,2)$ acts on $\Sn^1$ transitively.
 If $p\notin V^\perp$, dividing $p$ by $\sqrt{z^2+u^2}$ we may assume
\[-t^2+x^2+y^2=-z^2-u^2=-1,\]
which describes the direct product $\Hn^2\times \Sn^1$. Note that, this identification is $SO_\circ(1,2)\times SO(2)$-equivariant where the action of $SO_\circ(1,2)\times SO(2)$ on $\Hn^2\times \Sn^1$ is the one induced from the direct product structure. This shows that, $SO_\circ(1,2)\times SO(2)$ acts on the complement $\Sn^3\setminus \Sn^1$ transitively.
\end{itemize}

\begin{remark}\label{rem}
It is well-known that, the group $SO_\circ(1,2)$ has a unique $2$-dimensional connected Lie subgroup $\Aff$ called affine group. It consists of parabolic and hyperbolic elements and preserves a lightlike line in $\R^{1,2}$. Also, $SO_\circ(1,2)$, has three distinct $1$-parameter subgroups; an elliptic (resp. parabolic, hyperbolic) subgroup $\mathcal{E}$ (resp. $\mathcal{H}$, $\mathcal{P}$) preserving a unique timelike line (resp. a unique lightlike line, two distinct lightlike lines). It follows that every connected proper Lie subgroup of $SO_\circ(1,2)\times SO(2)$ with $\dim \geq 2$, is conjugate to $SO_\circ(1,2)$ or $\mathcal{E}\times SO(2)$, or it is a subgroup of $\Aff\times SO(2)$, up to conjugacy.
\end{remark}
Now, we prove Theorems \ref{thm.p.1} and \ref{thm.p.0}.

\textbf{Proof of Theorem \ref{thm.p.1}.}
By Theorem \ref{thm.1}, $G$ acts on $\R^{1,4}$ reducibly, since $SO_\circ(1,4)$ acts on $\Sn^3$ transitively. Let $V\leq \R^{1,4}$ be a non-trivial $G$-invariant subspace. The group $G$ fixes no point neither in $\Hn^4$ nor in $\Sn^3$. Thus, by Lemma \ref{lem.sub}, $V$ (or its orthogonal $V^\perp$) is a positive definite subspace of dimension 1 or 2.

If $\dim V=2$, then $G$ is a subgroup of $SO_\circ(1,2)\times SO(2)$, up to conjugacy. So, it preserves a great circle $\Sn^1\subset \Sn^3$. By Remark \ref{rem}, the only possible case for $G$ is $SO_\circ(1,2)$, since $SO_\circ(1,2)\times SO(2)$ acts on the complement $\Sn^3\setminus \Sn^1\approx \Hn^2\times \Sn^1$ transitively, and the other subgroups admit a fixed point either in $\Hn^4$ or in $\Sn^3$. The group $SO_\circ(1,2)$ acts on the great circle $\Sn^1$ transitively. Also, it acts on the $\Hn^2$-factor (resp. $\Sn^1$-factor) of $\Hn^2\times \Sn^1$ transitively (resp. trivially). Therefore, it admits a codimension $1$ foliation in $\Hn^2\times \Sn^1$. It follows that $G=SO_\circ(1,2)$, up to conjugacy.

If $V$ is $1$-dimensional, then $G$ is a subgroup of $SO_\circ(1,3)$, up to conjugacy. Observe that $SO_\circ(1,3)$ acts on $\Sn^3$ with cohomogeneity zero. So, by Theorem \ref{thm.1}, $G$ acts on $\R^{1,3}$ reducibly. Using the same argument as we used in Lemma \ref{lem.sub}, one can show $G=SO_\circ(1,2)$ (which is studied in the previous paragraph).
\hfill$\square$

\textbf{Proof of Theorem \ref{thm.p.0}:}
If $G$ acts on $\R^{1,4}$ irreducibly, then by Theorem \ref{thm.1}, $G=SO_\circ(1,4)$ which acts on $\Sn^3$ transitively. If $G$ acts on $\R^{1,4}$ reducibly, using the same argument as we used in the proof of Theorem \ref{thm.p.1}, one can show that $G$ is conjugate either to $SO_\circ(1,2)\times SO(2)$ or $SO_\circ(1,3)$.\hfill$\square$

\bibliographystyle{amsplain}

\end{document}